\documentclass[final,3p]{elsarticle}
 \usepackage{graphics}
 \usepackage{graphicx}
 \usepackage{epsfig}
\usepackage{amssymb}
 \usepackage{amsthm}
 \usepackage{lineno}
 \usepackage{amsmath}
   \numberwithin{equation}{section}
\usepackage{mathrsfs}

\newtheorem{thm}{Theorem}[section]

\newtheorem{lem}[thm]{Lemma}

\newtheorem{defn}[thm]{Definition}

\biboptions{sort&compress,square}
\begin{document}
\begin{frontmatter}
\author{Yuchen Yang}
\ead{yangyc580@nenu.edu.cn}
\author{Yong Wang\corref{cor2}}
\ead{wangy581@nenu.edu.cn}
\cortext[cor2]{Corresponding author.}
\address{School of Mathematics and Statistics, Northeast Normal University,
Changchun, 130024, P.R.China}

\title{ Statistical de Rham Hodge operators, spectral Einstein functionals\\ and the noncommutative residue}
\begin{abstract}
 Inspired by statistical de Rham Hodge operators and the spectral functionals,
 we carry on some promotion to spectral functionals to noncommutative fields,
 and  associate them with  the noncommutative residue on manifolds with boundary.
 We prove the Dabrowski-Sitarz-Zalecki type theorem for statistical de Rham Hodge operators on manifolds with boundary.
\end{abstract}
 
\end{frontmatter}
\section{Introduction}
\label{1}
Tracing the origin of noncommutative residue,  which is found in \cite{Gu,Wo1} by M. Wodzicki and V. W. Guillemin and it has very important significance in the study of noncommutative geometry.
On the even-dimensional Riemannian manifold, Dabrowski etc. defined the vector field and the bilinear functional of the differential form,
and their densities gives the Einstein tensor and metric.
 Furthermore, they extended it to noncommutative geometry. 
 and proved that the Einstein functional vanishes for the conformal rescaling geometry of noncommutative two torus \cite{DL}.
An notable spectral scheme is the small-time asymptotic expansion of the (localised) trace of heat kernel \cite{PBG,FGV}
which generates geometric objects such as residue, scalar curvature, and other scalar combinations of curvature tensors on manifolds.
The theory has a very rich structure and research significance in physics and mathematics.

 In \cite{Ac,Wo,Wo1,Gu},  Ackermann, Wodzicki and Guillemin syudied noncommutative residue and obtained some results
and then they got, on the algebra of pseudo-differential operators $\Psi DO(E)$, the linear functional $res: \Psi DO(E)\rightarrow \mathbb{C }$ is actually the unique trace (up to multiplication by constants).
Tracing the origin of noncommutative residue, Connes  derived a conformal 4-dimensional Polyakov action analogue in \cite{Co1}, and a very valuable observation was proposed in \cite{Co2}.
More accurately, an explosive observation that the Wodzicki residue of the inverse square of the Dirac operator is equal to the
  Einstein-Hilbert action of general relativity was  proposed by Connes \cite{Co2,Co3}.
  Since Kastler \cite{Ka},  Kalau and Walze \cite{KW} have proved it respectively, we call this conjecture as the Kastler-Kalau-Walze theorem.
  In the Boutet de Monvel calculus, the continuous trace on the classical finite element algebra is constructed by Wodzicki, Fedosov etc.\cite{FGLS} on a compact manifold $M (dim M>2$) with boundary .
  If the boundary is reduced to an empty set, this trace is consistent with the Wodzicki's noncommutative residue.
It is a natural way to study the Kastler-Kalau-Walze type theorem of manifolds with boundary and the operator theory interpretation of gravitational action by using elliptic pseudo-differential operators and noncommutative residues.
  Wang provided an effective method for the noncommutative residue for Dirac operators and signature operators and proved Kastler-Kalau-Walze type theorems for manifolds with boundary \cite{Wa1,Wa3,Wa4}, which can offer help for our research.

 In 1989, Jean-Michel Bismut established a local index theorem for Dirac operators when $TM$ is equipped with a general
 Euclidean connection and specialized it to complex manifold  \cite{JMB}.
In 1996, Ackermann and Tolksdorf \cite{AT} used the generalized version of the distinguished Lichnerowicz formula for the square of the
most general Dirac operator to computed the subleading term $\Phi_1(x,x,\tilde{D})$ of the heat-kernel expansion of $\tilde{D}^2$.
 Then to 2012, Pf$\ddot{a}$ffle and Stephan \cite{PS} studied the induced Dirac operators
and considered compact Riemannian spin manifolds without boundary equipped with orthogonal connections.
Wang, Wang and Wu computed $\widetilde{wres}[\pi^+\tilde\nabla_{X}\tilde\nabla_{Y}(D_T^*D_T)^{-1}\circ\pi^+(D^*_TD_T)^{-1}]$ and
$\widetilde{wres}[\pi^+\tilde\nabla_{X}\tilde\nabla_{Y}D_T^{-1}\circ\pi^+(D^*_TD_TD^*_T)]$ in \cite{WWw}, and  provided an important reference for our paper calculation.
Wang and Wang proved the equivariant Dabrowski-Sitarz-Zalecki
type theorems of Dirac operators for lower dimensional  manifolds with (or without) boundary  \cite{WW3} and the  Kastler-Kalau-Walze type theorem of sub-Dirac operators on lower dimension  is calculated \cite{WW2}.
 In \cite{wswy}, Wei and Wang gave the Lichnerowicz type formulas for statistical de Rham Hodge operators and they proved
Kastler-Kalau-Walze type theorems on compact manifolds with boundary.
On this basis, Li, Wu and Wang  promoted it to $\widetilde{wres}(\pi^+\tilde{D}^{-1}\circ \pi^+\tilde{D}^{*-1}(\tilde{D^*}\tilde{D})^{-2m})\}$  on $\wedge^*T^*M$ for $2m+4$-dimensional orientable Riemannian manifold in \cite{lww}.
On low-dimensional manifolds with boundary, We  proved Dabrowski-Sitarz-Zalecki type theorems for the spectral Einstein functional associated with the Dirac operator in \cite{yw}.

This paper is a generalization of the results in \cite{DL}, \cite{WWw}, \cite{wswy},
the purpose  is to get the spectral functional
 associated with statistical de Rham Hodge operators on compact manifolds with  boundary
 and compute the  residue of $\widetilde{\nabla}_{U}\widetilde{\nabla}_{V}T^{-4}$  so as to obtain  Dabrowski-Sitarz-Zalecki type theorems.

\section{Spectral functionals for the statistical de Rham Hodge operator with Torsion}
 In this section,  we consider $(M, g^M)$ is an n-dimensional $(n>3)$ oriented Riemannian manifold.
 Let an elliptic differential operator $d$ is the de Rham derivative on $C^{\infty}(M; \wedge^*T^*M)$.
Then we get the de Rham coderivative $\delta=d^*$ and the symmetric operators $\widehat T=d+\delta$. 
We use $\Delta=\delta d+d\delta$ to define the standard Hodge Laplacian.

By \cite{wswy}  and \cite{BO}, we have
\begin{equation}
\Delta^{\widehat\nabla}=(\delta-\iota(E_0)+d)^2,
\end{equation}
where $\iota(E_0)$ is the contraction operator and $E$ is a vector field.\\
For $V'=\Gamma(TM)$, we have the generalized statistical de Rham Hodge operators 
\begin{align}
T_i&=d+\delta+\lambda_i\iota(V'), (i=1,2); \nonumber\\
T_i^*&=d+\delta+\lambda_i\epsilon(V^{*'}), (i=1,2;~\lambda_i\in \mathbb{R}).
\end{align}
 We mainly calculate the case of $T_i$, and let $\lambda_i=1, i=1,2$.
Proved by the \cite{wswy} , let $g^{ij}=g(dx_i,dx_j), \xi=\sum_k\xi_jdx_j$, $\omega_{s,t}(e_i):=\langle \nabla^{TM}_{e_i}e_t, e_s\rangle$ and $\nabla^L_{\partial_i}\partial_j=\sum_k\Gamma^k_{ij}\partial_k$,
 the statistical de Rham Hodge operators $T_i$ and $T_i^*$ can be written as
\begin{align}
T_i&=\sum_{i=1}^nc(e_i)[e_i+a_i+\sigma_i]+\lambda_i\iota(V'), (i=1,2); \nonumber\\
T_i^*&=\sum_{i=1}^nc(e_i)[e_i+a_i+\sigma_i]+\lambda_i\epsilon(V^{*'}), (i=1,2),
\end{align}
where
\begin{align}
\sigma_i=-\frac{1}{4}\sum_{s,t}\omega_{s,t}(e_i)c(e_s)c(e_t);~~ a_i=\frac{1}{4}\sum_{s,t}\omega_{s,t}(e_i)\overline c(e_s)\overline c(e_t); \nonumber\\
\xi^j=g^{ij}\xi_i;~~\Gamma^k=g^{ij}\Gamma^k_{ij};~~\sigma^j=g^{ij}\sigma_i;~~ a^j=g^{ij}a_i.
\end{align}

From  theorem 2.1 in \cite{wswy}, we have the following Lichneriowicz formulas:
\begin{lem}\label{lem01}\cite{PS}
When $\lambda_1=1$, for the  statistical de Rham Hodge  operator $T$, we have
\begin{align}
T^2=&-[g^{ij}(\nabla_{\partial_i}\nabla_{\partial_j}-\nabla_{\nabla_{\partial_i}^L\partial_j})]-\frac{1}{8}\sum_{ijkl}R_{ijkl}\overline c(e_i)\overline c(e_j)c(e_k)c(e_l)+\frac{1}{4}s\nonumber\\
&+\frac{1}{4}\sum_i[c(e_i)\iota(V')+\iota(V')c(e_i)]^2-\frac{1}{2}[\nabla^{TM}_{e_j}(\iota(V'))c(e_j)-c(e_j)\nabla^{TM}_{e_j}(\iota(V'))],
\end{align}
where $s$ is the scalar curvature.
\end{lem}

In this paper, all s will represent the scalar curvature and will not be repeated later.
For the Einstein functional we want to prove the following lemma of Einstein functional proposed by Dabrovsky et al.is very important.
On a compact Riemannian manifold $M$ of dimension $n = 2m$, let $U$, $V$ be a pair of vector fields. 
 By using the fact that the Laplace operator $\Delta^{-1}=T^2=\Delta+E $ acting on the cross section of a vector bundle $\wedge^*T^*M$  of rank $2^n$
  can contain some nontrivial connections and torsions, the spectral functionals on vector fields are defined.
 
\begin{lem}\cite{DL}
The Einstein functional equal to
 \begin{equation}
Wres\big(\widetilde{\nabla}_{U}\widetilde{\nabla}_{V}\Delta^{-m}\big)=\frac{\upsilon_{n-1}}{6}2^{n}\int_{M}G(U,V)vol_{g}
 +\frac{\upsilon_{n-1}}{2}\int_{M}F(U,V)vol_{g}+\frac{1}{2}\int_{M}(\mathrm{Tr}E)g(U,V)vol_{g},
\end{equation}
where $G(U,V)$ denotes the Einstein tensor evaluated on the two vector fields, $F(U,V)=Tr(U_{a}V_{b}F_{ab})$ and
$F_{ab}$ is the curvature tensor of the connection, $\mathrm{Tr}E$ denotes the trace of $E$ and $\upsilon_{n-1}=\frac{2\pi^{m}}{\Gamma(m)}$.
\end{lem}
The aim of this section is to prove the following.
\begin{thm}\label{thm3}
For the Laplace (type) operator $\Delta=T^2$, the Einstein functional equal to
\begin{align}\label{w}
Wres\big(\widetilde{\nabla}_{U}\widetilde{\nabla}_{V}\Delta^{-n}\big)
=&\frac{2^{2m+1}\pi^{m}}{6\Gamma(m)}\int_{M}\big(Ric(U,V)-\frac{1}{2}sg(U,V)\big) vol_{g}\nonumber\\
&-\frac{\pi^{m}}{\Gamma(m)}\int_{M}\big(\frac{1}{2}[g(V, \widetilde\nabla_{U}V')+g(U,\widetilde\nabla_{V}V')]\mathrm{Tr}^{\wedge^*(T^*M)}[Id]\big) vol_{g}\nonumber\\
&+\int_{M}2^{2m-1}
\big( \frac{1}{4}s\mathrm{Tr}[Id]+\frac{1}{2}|V'|^2\big)g(U,V) vol_{g},
\end{align}
where $V'$ be a vector bundle on $M$ and $\widetilde{\nabla}_{U}^{V'}=\nabla_{U}^{\wedge^*(T^*M)}-\frac{1}{2}g(U,V')Id^{\wedge^*(T^*M)}$.
\end{thm}

\begin{proof}
	By the definition of connection $\widetilde\nabla^V$, we have
 \begin{align}
\widetilde{\nabla}_{U}^{V'}\psi=&\nabla_{U}^{\wedge^*(T^*M)}\psi-\frac{1}{2}g(U,V')Id^{\wedge^*(T^*M)}\psi\nonumber\\
=&U\psi+\sigma(U)\psi-\frac{1}{2}g(U,V')Id^{\wedge^*(T^*M)} \psi\nonumber\\
:=&U\psi+\overline{J}(U)\psi,
\end{align}
where
 \begin{equation}
\sigma(U)=\frac{1}{4}\sum_{s,t} \omega_{s,t}(U) [\overline c(e_s)\overline c(e_t)-c(e_s)c(e_t)].
\end{equation}
Let $U=\sum_{a=1}^{n}U^{a}e_{a}$, $V=\sum_{b=1}^{n}V^{b}e_{b}$,
in view of that
 \begin{equation}
F(U,V)=Tr(U_{a}V_{b}F_{ab})=\sum_{a,b=1}^{n}U^{a}V^{b}Tr^{\wedge^*(T^*M)}(F_{e_{a},e_{b}}).
\end{equation}
By \cite{,WWw}, we obtain
 \begin{align}
F_{e_{a},e_{b}}
=e_{a}(\overline{J}(e_{b}))-e_{b}(\overline{J}(e_{a}))+\overline{J}(e_{a})\overline{J}(e_{b})
-\overline{J}(e_{b})\overline{J}(e_{a})-\overline{J}([e_{a},e_{b}]).
\end{align}
Also, straightforward computations yield
 \begin{align}
\mathrm{Tr}^{\wedge^*(T^*M)}\big(e_{a}(\overline{J}(e_{b}))\big)
=&\mathrm{Tr}^{\wedge^*(T^*M)}\Big[e_{a}(\frac{1}{4}\sum_{s,t} \omega_{s,t}(e_{b})) (\overline c(e_s)\overline c(e_t)-c(e_s)c(e_t))-\frac{1}{2}g(e_b,V')\Big]\nonumber\\
=&\mathrm{Tr}^{\wedge^*(T^*M)}\Big[\frac{1}{4}\sum_{s,t} e_{a}(\omega_{s,t}(e_{b})) (\overline c(e_s)\overline c(e_t)-c(e_s)c(e_t))-\frac{1}{2}e_a(g(e_b,V'))\Big]\nonumber\\
=&-\frac{1}{2}e_a(g(e_b,V'))\mathrm{Tr}^{\wedge^*(T^*M)}[Id],
\end{align}
so we have
\begin{align}
	-\mathrm{Tr}^{\wedge^*(T^*M)}[e_b\overline{J}(e_a)]=\frac{1}{2}e_b(g(e_a,V')\mathrm{Tr}^{\wedge^*(T^*M)}[Id].
\end{align}
Next
\begin{align}
	&\mathrm{Tr}^{\wedge^*(T^*M)}(\overline{J}([e_a, e_b]))\nonumber\\
	=&\mathrm{Tr}^{\wedge^*(T^*M)}[\frac{1}{4}\sum_{s,t}\omega_{s,t}([e_a, e_b])(\overline c(e_s)\overline c(e_t)-c(e_s)c(e_t))-\frac{1}{2}g([e_a, e_b],V)']\nonumber\\
	=&-\frac{1}{2}g([e_a, e_b],V')\mathrm{Tr}[Id],
\end{align}
\begin{align}
	&\mathrm{Tr}^{\wedge^*(T^*M)}(F_{e_{a},e_{b}})\nonumber\\
	&=-\frac{1}{2}e_a(g(e_b,V'))\mathrm{Tr}^{\wedge^*(T^*M)}[Id]+\frac{1}{2}e_b(g(e_a,V')\mathrm{Tr}^{\wedge^*(T^*M)}[Id]+\frac{1}{2}g([e_a, e_b],V')\mathrm{Tr}[Id]\nonumber\\
	=&-\frac{1}{2}[g(\nabla_{e_a}^Le_b, V')+g(e_b,\nabla_{e_a}V')]
	+\frac{1}{2}[g(\nabla_{e_b}^Le_a, V')+g(e_a,\nabla_{e_b}V')]
	+\frac{1}{2}g([e_a, e_b], V')\nonumber\\
	=&-\frac{1}{2}[g(e_b, \nabla_{e_a}V')+g(e_a, \nabla_{e_b}V')]\mathrm{Tr}^{\wedge^*(T^*M)}[Id],
\end{align}
therefore we get
\begin{align}\label{f}
	F(U,V)&=\sum_{a,b=1}^{n}U^{a}V^{b}Tr^{\wedge^*(T^*M)}(F_{e_{a},e_{b}})\nonumber\\
	&=-\frac{1}{2}[g(V, \nabla_{U}V')+g(U,\nabla_{V}V')]\mathrm{Tr}^{\wedge^*(T^*M)}[Id].
\end{align}
Let $\Delta=\Delta_{0}+E$, by (2.35) in \cite{wswy}, we have
 \begin{align}
E=&\frac{1}{8}\sum_{ijkl}R_{ijkl}\overline c(e_i)\overline c(e_j)c(e_k)c(e_l)-\frac{1}{4}s-\frac{1}{4}\sum_{i}[c(e_i)\iota(V')+\iota(V')c(e_i)]^2\nonumber\\
		&+\frac{1}{2}[\nabla_{e_j}^{TM}(\iota(V'))c(e_j)-c(e_j)\nabla_{e_j}^{TM}(\iota(V'))].
	\end{align}
Now, we can get $\mathrm{Tr}^{\wedge^*T^*M}(E)$
\begin{align}\label{e}
\mathrm{Tr}^{\wedge^*T^*M}(E)
=\frac{1}{4}s\mathrm{Tr}[Id]+\frac{1}{2}|V'|^2\mathrm{Tr}[Id].
\end{align}
Summing up \eqref{f}-\eqref{e} leads to the desired equality \eqref{w}, and the proof of
the Theorem is complete.
\end{proof}

So locally we can use Theorem \ref{thm3} to compute
\begin{thm}
	Let M be a 4-dimensional compact manifold without boundary and $\widetilde{\nabla}$ be a connection in Theorem \ref{thm3}. Then we get the 
	 spectral Einstein functional on compact manifolds without boundary
	\begin{align}
		&Wres[\sigma_{-4}(\widetilde{\nabla}_{U}\widetilde{\nabla}_{V}T^{-4})
		]\nonumber\\
		=&\frac{4\pi^{2}}{3}\int_{M}\big(Ric(U,V)-\frac{1}{2}sg(U,V)\big) vol_{g}
		-\pi^2\int_{M}\big(\frac{1}{2}[g(V, \nabla_{U}V')+g(U,\nabla_{V}V')]\mathrm{Tr}^{\wedge^*(T^*M)}[Id]\big) vol_{g}\nonumber\\
		&+\int_{M}
		\big( 2s\mathrm{Tr}[Id]+4|V'|^2\big)g(U,V) vol_{g}.
	\end{align}

\end{thm}

 \section{Residue for Statistical de Rham Hodge Operators $\widetilde{\nabla}_{U}\widetilde{\nabla}_{V}T^{-2}$ and $T^{-2}$ }

In this section, we compute the  spectral Einstein functional for 4-dimension compact manifolds with boundary and get a
Dabrowski-Sitarz-Zalecki type formula in this case.
Some basic knowledge such as boundary metric, frame, noncommutative residue of manifold with boundary and Boutet de Monvels algebra can be referred to in \cite{Wa1}, and we will not repeat it here.
 We will consider $T^{2}$.
Let $q_{1},q_{2}$ be nonnegative integers and $q_{1}+q_{2}\leq n$,
denote by $\sigma_{\mu}(P)$ the $\mu$-order symbol of an operator $P$,
an application of (3.5) and (3.6) in \cite{Wa1} shows that
\begin{defn}\label{defn1}
	 Spetral Einstein functional associated the statistical de Rham Hodge Operators $T$ of manifolds with boundary are defined by
	\begin{equation}\label{}
		Vol_{n}^{\{q_{1},q_{2}\}}M:=\widetilde{Wres}[\pi^{+}(\widetilde{\nabla}_{U}\widetilde{\nabla}_{V}(T^2)^{- q_{1}})
		\circ\pi^{+}(T^{-2})^{q_{2}}].
	\end{equation}
	where $\pi^{+}(\widetilde{\nabla}_{U}\widetilde{\nabla}_{V}(T^2)^{- q_{1}})$, $\pi^{+}(T^{-2})^{q_{2}}$ are
	elements in Boutet de Monvel's algebra\cite{Wa3}.
\end{defn}\
 So we only need to compute $\int_{\partial M}\Phi$. For statistical de Rham Hodge operators
  $\widetilde{\nabla}_{U}\widetilde{\nabla}_{V}T^{-2}$ and $T^{-2}$,
denote by $\sigma_{\mu}(P)$ the $\mu$-order symbol of an operator $P$. An application of (2.1.4) in \cite{Wa1} shows that
\begin{align}
	&\widetilde{Wres}[\pi^{+}(\widetilde{\nabla}_{U}\widetilde{\nabla}_{V}T^{-2})^{q_{1}}
	\circ\pi^{+}(T^2)^{-q_{2}}]\nonumber\\
	&=\int_{M}\int_{|\xi|=1}\mathrm{Tr}_{S(TM)}
	[\sigma_{-n}(\widetilde{\nabla}_{U}\widetilde{\nabla}_{V}(T^2)^{-q_{1}}
	\circ (T^2)^{-q_{2}})]\sigma(\xi)\texttt{d}x+\int_{\partial M}\Phi,
\end{align}
where
\begin{align}\label{a4}
	\Phi=&\int_{|\xi'|=1}\int_{-\infty}^{+\infty}\sum_{j,k=0}^{\infty}\sum \frac{(-i)^{|\alpha|+j+k+\ell}}{\alpha!(j+k+1)!}
	\mathrm{Tr}_{S(TM)}[\partial_{x_{n}}^{j}\partial_{\xi'}^{\alpha}\partial_{\xi_{n}}^{k}\sigma_{r}^{+}
	((\widetilde{\nabla}_{U}\widetilde{\nabla}_{V}(T^2)^{-q_{1}})(x',0,\xi',\xi_{n})\nonumber\\
	&\times\partial_{x_{n}}^{\alpha}\partial_{\xi_{n}}^{j+1}\partial_{x_{n}}^{k}\sigma_{l}((T^2)^{-q_{2}})(x',0,\xi',\xi_{n})]
	\texttt{d}\xi_{n}\sigma(\xi')\texttt{d}x' ,
\end{align}
and the sum is taken over $r-k+|\alpha|+\ell-j-1=-n,r\leq-q_{1},\ell\leq-q_{2}$.

 Recall the definition of the statistical de Rham Hodge operator T in \cite{wswy}:
\begin{equation}
T=\sum_{i=1}^{n}c({e_{i}})\{{e_{i}}+\frac{1}{4}\sum_{s,t}\omega_{s,t}({e_{i}})[\hat{c}(e_s)\hat{c}(e_t)-c({e_{s}})c({e_{t}})]\}+\iota(V'),
\end{equation}
where $c(e_{i})$ denotes the Clifford action.
We define connection $\nabla_U^{\wedge^*(T^*M),V'}:=U+\frac{1}{4}\sum_{ij}\langle\nabla_U^L{e_i},e_j\rangle [c(e_i)c(e_j)-\hat{c}(e_i)\hat{c}(e_j)]-\frac{1}{2}g(V',U)Id^{\wedge^*(T^*M)}$. Set
\begin{equation}
A(U)=\frac{1}{4}\Sigma_{ij}\langle\nabla_U^L{e_i},e_j\rangle [c(e_i)c(e_j)-\hat{c}(e_i)\hat{c}(e_j)];
~~B(U)=\frac{1}{2}g(V',U)Id^{\wedge^*(T^*M)}
\end{equation}
 Let $\widetilde{\nabla}_{U}=U+A(U)+B(U)$ and
 $\widetilde{\nabla}_{V}=V+A(V)+B(V)$,  we obtain
\begin{align}
\widetilde{\nabla}_{U}\widetilde{\nabla}_{V}&=(U+A(U)
+B(U))(V+A(V)+B(V))\nonumber\\
 &=UV+U[A(V)]+A(V)U+A(U)V+A(U)A(V)+B(U)V
 +B(U)A(V)\nonumber\\
 &~~~~+U[B(V)]+B(V)U+A(U)B(V)+B(U)B(V),
\end{align}
where
$U=\Sigma_{j=1}^nU_j\partial_{x_j}, V=\Sigma_{l=1}^nV_l\partial_{x_l}$.

 Let $g^{ij}=g(dx_{i},dx_{j})$, $\xi=\sum_{k}\xi_{j}dx_{j}$
  and $\nabla^L_{\partial_{i}}\partial_{j}=\sum_{k}\Gamma_{ij}^{k}\partial_{k}$,
   we get
\begin{align}
&\sigma_{j}=-\frac{1}{4}\sum_{s,t}\omega_{s,t}
(e_j){[\hat{c}(e_s)\hat{c}(e_t)-c(e_s)c(e_t)]}
;~~~\xi^{j}=g^{ij}\xi_{i};~~~~\Gamma^{k}=g^{ij}\Gamma_{ij}^{k};~~~~\sigma^{j}=g^{ij}\sigma_{i}.
\end{align}
Then we have the following lemmas.
\begin{lem}\label{lem3} The following identities hold:
\begin{align}
 \sigma_{0}(\widetilde{\nabla}_{U}\widetilde{\nabla}_{V})=&U[A(V)]+U[B(V)]+A(U)A(V)
+B(U)A(V)+A(U)B(V)+B(U)B(V);\\
\sigma_{1}(\widetilde{\nabla}_{U}\widetilde{\nabla}_{V})
=&\sqrt{-1}\sum_{j,l=1}^nU_j\frac{\partial_{V_l}}{\partial_{x_j}}\sqrt{-1}\xi_l
+\sqrt{-1}\sum_jA(V)U_j\xi_j+\sqrt{-1}\sum_lA(V)V_l\xi_l\nonumber\\
&+\sum_j B(U)V_j\sqrt{-1}\xi_j+\sum_j B(V)U_j\sqrt{-1} \xi_j;\\
\sigma_{2}(\widetilde{\nabla}_{U}\widetilde{\nabla}_{V})=&-\sum_{j,l=1}^nU_jV_l\xi_j\xi_l.
\end{align}
\end{lem}

By  Lemma \ref{lem01} and $\sigma(q_{1}\circ q_{2})=\sum_{\alpha}\frac{1}{\alpha!}\partial^{\alpha}_{\xi}[\sigma(q_{1})]
DT_x^{\alpha}[\sigma(q_{2})]$, then
\begin{lem} \label{lem4}The following identities hold:
\begin{align}
	\sigma_{-1}(T^{-1})&=\frac{ic(\xi)}{|\xi|^2},\nonumber\\
	\sigma_{-2}(T^{-1})&=\frac{c(\xi)\sigma_{0}(T)c(\xi)}{|\xi|^4}
	+\frac{c(\xi)}{|\xi|^6}\sum_{j}c(\mathrm{d}x_j)[\partial_{x_j}(c(\xi))|\xi|^2-c(\xi)\partial_{x_j}(|\xi|^2)];\\
\sigma_{-2}(T^{-2})&=|\xi|^{2}=\sigma_{-2}(D^{-2})=|\xi|^{-2},\nonumber\\
\sigma_{-3}(T^{-2})&=-\sqrt{-1}|\xi|^{-4}\xi_k(\Gamma^k-2\delta^k-2a^j)
-\sqrt{-1}|\xi|^{-6}2\xi^j\xi_\alpha\xi_\beta\partial_jg^{\alpha\beta}\nonumber\\
&-\sqrt{-1}\big(c(\xi)\iota(v')+\iota(v')c(\xi)  \big)|\xi|^{-4},
\end{align}
where,
\begin{align}
	\sigma_{0}(T)&=\frac{1}{4}\sum_{i,s,t}\omega_{s,t}(e_i)c(e_i)\hat{c}(e_s)\hat{c}(e_t)
	-\frac{1}{4}\sum_{i,s,t}\omega_{s,t}(e_i)c(e_i)c(e_s)c(e_t)+\iota(V')\nonumber.
\end{align}	
\end{lem}

\begin{lem} The following equation holds :
\begin{align}
\sigma_{0}(\widetilde{\nabla}_{U}\widetilde{\nabla}_{V}(T^{-2})=&
-\sum_{j,l=1}^nU_jVl\xi_j\xi_l|\xi|^{-2};\\
\sigma_{-1}(\widetilde{\nabla}_{U}\widetilde{\nabla}_{V}(T^{-2})=&
\sigma_{2}(\widetilde{\nabla}_{U}\widetilde{\nabla}_{V})\sigma_{-3}(T^{-2})
+\sigma_{1}(\widetilde{\nabla}_{U}\widetilde{\nabla}_{V})\sigma_{-2}((T{-2})\nonumber\\
&+\sum_{j=1}^{n}\partial_{\xi_{j}}\big[\sigma_{2}(\widetilde{\nabla}_{U}\widetilde{\nabla}_{V})\big]
T_{x_{j}}\big[\sigma_{-2}((T^{-2})\big].
\end{align}
\end{lem}

 By Lemma 2.2 in \cite{Wa3}, we have

\begin{lem}\label{le:32}
With the metric $\frac{1}{h(x_{n})}g^{\partial M}+\texttt{d}x _{n}^{2}$ on $M$ near the boundary
\begin{equation*}
	\partial x_{i}(|\xi|^{2}_{g^{M}})(x_{0})=
	\begin{cases}
		0,~~&\text{if}~~i<n,\\
		h'(0)|\xi'|^{2}_{g^{\partial M}},~~&\text{if}~~i=n.
	\end{cases}
\end{equation*}
\begin{equation*}
	\partial x_{i}(c(\xi))(x_{0})=
	\begin{cases}
		0,~~&\text{if}~~i<n,\\
		\partial x_{n}(c(\xi'))(x_{0}),~~&\text{if}~~i=n.
	\end{cases}
\end{equation*}
where $\xi=\xi'+\xi_{n}\texttt{d}x_{n}$.
\end{lem}

 $\int_{\partial M} \Phi$ is calculated below.  When $n=4$, we have ${\rm tr}_{\wedge^*T^*M}[{\rm \texttt{id}}]={\rm dim}(\wedge^*(\mathbb{R}^2))=16$, the sum is taken over $
r+l-k-j-|\alpha|=-3,~~r\leq 0,~~l\leq-2,$ so we have the following five cases:

Similar to \cite{WWw},  the calculation results of {\bf case a)} are as follows.
\begin{align}\label{35}
	case~a-I):~~~\Phi_1&=0;\nonumber\\
case~a-II):~~
\Phi_2&=\left(\frac{13\pi}{6}\sum_{j=1}^{n-1}U_jV_j+\frac{13}{8}U_nV_n\right)h'(0)\pi\Omega_3dx';\nonumber\\
case~a-III):~\Phi_3
&=\left(\frac{5\pi}{3}\sum_{j=1}^{n-1}U_jV_j+\frac{5i}{4}U_nV_n\right)h'(0)\pi\Omega_3dx'.
\end{align}

 {\bf case b)}~$r=0,~l=-3,~k=j=|\alpha|=0$.

 By  \eqref{a4}, we have
\begin{align}\label{42}
\Phi_4&=-i\int_{|\xi'|=1}\int^{+\infty}_{-\infty}\mathrm{Tr} [\pi^+_{\xi_n}
\sigma_{0}(\widetilde{\nabla}_{U}\widetilde{\nabla}_{V}(T^{-2})\times
\partial_{\xi_n}\sigma_{-3}(T^{-2})](x_0)d\xi_n\sigma(\xi')dx'\nonumber\\
&=i\int_{|\xi'|=1}\int^{+\infty}_{-\infty}\mathrm{Tr} [\partial_{\xi_n}\pi^+_{\xi_n}
\sigma_{0}(\widetilde{\nabla}_{U}\widetilde{\nabla}_{V}(T^{-2})\times
\sigma_{-3}(T^{-2})](x_0)d\xi_n\sigma(\xi')dx'.
\end{align}
 By Lemma \ref{lem4}, we get
\begin{align}\label{43}
\sigma_{-3}(T^{-2})(x_0)|_{|\xi'|=1}
&=-\frac{i}{(1+\xi_n^2)^2}\left(c(\xi)\iota(V')+\iota(V')c(\xi)\right)-\frac{2ih'(0)\xi_n}{(1+\xi_n^2)^3}\nonumber\\
&-\frac{h'(0)}{4(1+\xi_n^2)^2}[\sum_{t<n}(\hat{c}(e_n)\hat{c}(e_t)-c(e_n)c(e_t))+2].
\end{align}
\begin{align}\label{45}
\partial_{\xi_n}\pi^+_{\xi_n}\sigma_{0}(\widetilde{\nabla}_{U}\widetilde{\nabla}_{V}(T^{-2})(x_0)|_{|\xi'|=1}
&=-\frac{i}{2(\xi_n-i)^2}\sum_{j,l=1}^{n-1}U_jV_l\xi_j\xi_l-\frac{1}{2(\xi_n-i)^2}U_nV_n\nonumber\\
&+\frac{1}{2(\xi_n-i)^2}\sum_{j=1}^{n-1}U_jV_n\xi_j+\frac{1}{2(\xi_n-i)^2}\sum_{l=1}^{n-1}U_nV_l\xi_l.
\end{align}
We note that $\int_{|\xi'|=1}\xi_{i_{1}}\xi_{i_{2}}\cdots\xi_{i_{2d+1}}\sigma(\xi')=0~i<n$,
therefore, we omit some terms that do not contribute to the calculation of {\bf case b)}.
\begin{align}
	\mathrm{Tr}[c(\xi)\iota(V')+\iota(V')c(\xi)]=\mathrm{Tr}[g(\xi,V')Id]=\left\langle \xi, V' \right\rangle\mathrm{Tr}[Id]
	=2^4\left\langle \xi', V' \right\rangle+2^4\left\langle dx_n, V' \right\rangle\xi_n.
\end{align}
Then, we have
\begin{align}\label{39}
&\mathrm{Tr}[\partial_{\xi_n}\pi^+_{\xi_n}\sigma_{0}(\widetilde{\nabla}_{U}\widetilde{\nabla}_{V}T^{-2})\times
\sigma_{-3}T^{-2})](x_0)\nonumber\\
&=\frac{2^3[\left\langle \xi', V' \right\rangle+\left\langle dx_n, V' \right\rangle\xi_n]}{(\xi_n-i)^4(\xi_n+i)^2}\sum_{j,l=1}^{n-1}U_jV_l\xi_j\xi_l
 -\frac{2^3i[\left\langle \xi', V' \right\rangle+\left\langle dx_n, V' \right\rangle\xi_n]U_nV_n}{(\xi_n-i)^4(\xi_n+i)^2}\nonumber\\
&-\frac{2^4ih'(0)\xi_n} {(\xi_n-i)^4(\xi_n+i)^2}\sum_{j,l=1}^{n-1}U_jV_l\xi_j\xi_l
 +\frac{2^4h'(0)\xi_n}{(\xi_n-i)^4(\xi_n+i)^2}U_nV_n\nonumber\\
&-\frac{2^4ih'(0)\xi_n}{(\xi_n-i)^5(\xi_n+i)^3}\sum_{j,l=1}^{n-1}U_jV_l\xi_j\xi_l
 -\frac{2^4ih'(0)\xi_n}{(\xi_n-i)^5(\xi_n+i)^3}U_nV_n.
\end{align}
Finally, we get
\begin{align}\label{41}
\Phi_4&=i\int_{|\xi'|=1}\int^{+\infty}_{-\infty}
\mathrm{Tr}[\partial_{\xi_n}\pi^+_{\xi_n}\sigma_{0}(\widetilde{\nabla}_{U}\widetilde{\nabla}_{V}T^{-2})\times
\sigma_{-3}T^{-2})](x_0)d\xi_n\sigma(\xi')dx'\nonumber\\
&=-i\sum_{j,l=1}^{n-1}U_jV_lh'(0)\Omega_3\int_{\Gamma^{+}}\left[\frac{2^3[\left\langle \xi', V'\right\rangle
	+\left\langle dx_n, V' \right\rangle\xi_n]}{(\xi_n-i)^4(\xi_n+i)^2}
-\frac{2^4ih'(0)\xi_n} {(\xi_n-i)^4(\xi_n+i)^2}\right.\nonumber\\
&\left.-\frac{2^4ih'(0)\xi_n}{(\xi_n-i)^5(\xi_n+i)^3}\right]\xi_j\xi_ld\xi_{n}dx'
+iU_nV_nh'(0)\Omega_3\int_{\Gamma^{+}}\left[ -\frac{2^3i[\left\langle \xi', V' \right\rangle+\langle dx_n, V' \rangle\xi_n]}{(\xi_n-i)^4(\xi_n+i)^2}\right.\nonumber\\
&\left.+\frac{2^4h'(0)\xi_n}{(\xi_n-i)^4(\xi_n+i)^2}
-\frac{2^4ih'(0)\xi_n}{(\xi_n-i)^5(\xi_n+i)^3}\right]d\xi_{n}dx'\nonumber\\
&=\left(\sum_{j=1}^{n-1}U_jV_j[-\frac{4\pi \langle dx_n, V'\rangle}{3}+\frac{10\pi ih'(0)}{3}]
-U_nV_n[-i\langle dx_n, V'\rangle+\frac{4-i}{2}h'(0)]\right)\pi\Omega_3dx'.
\end{align}

 {\bf  case c)}~$r=-1,~\ell=-2,~k=j=|\alpha|=0$.

By  \eqref{a4}, we get
\begin{align}\label{161}
\Phi_5=-i\int_{|\xi'|=1}\int^{+\infty}_{-\infty}\mathrm{Tr} [\pi^+_{\xi_n}
\sigma_{-1}(\widetilde{\nabla}_{U}\widetilde{\nabla}_{V}T^{-2})\times
\partial_{\xi_n}\sigma_{-2}T^{-2})](x_0)d\xi_n\sigma(\xi')dx'.
\end{align}
By Lemma \ref{lem4}, we have
\begin{align}\label{62}
\partial_{\xi_n}\sigma_{-2}(T^{-2})(x_0)|_{|\xi'|=1}=-\frac{2\xi_n}{(\xi_n^2+1)^2}.
\end{align}
Since
\begin{align}\label{2a31}
\sigma_{-1}(\widetilde{\nabla}_{U}\widetilde{\nabla}_{V}(T^{-2})(x_0)|_{|\xi'|=1}
=&\sigma_{2}(\widetilde{\nabla}_{U}\widetilde{\nabla}_{V})\sigma_{-3}(T^{-2})
+\sigma_{1}(\widetilde{\nabla}_{U}\widetilde{\nabla}_{V})\sigma_{-2}(T^{-2})\nonumber\\
&+\sum_{j=1}^{n}\partial_{\xi_{j}}\big[\sigma_{2}(\widetilde{\nabla}_{U}\widetilde{\nabla}_{V})\big]
T_{x_{j}}\big[\sigma_{-2}(T^{-2})\big].
\end{align}
 Explicit representation the first item of \eqref{2a31},
\begin{align}
&\sigma_{2}(\widetilde{\nabla}_{U}\widetilde{\nabla}_{V})\sigma_{-3}(T^{-2})(x_0)|_{|\xi'|=1}\nonumber\\
=&-\sum_{j,l=1}^{n}U_jV_l\xi_j\xi_l\times\Big(-\frac{i}{(1+\xi_n^2)^2}\left(c(\xi)\iota(v')+\iota(v')c(\xi)\right)-\frac{2ih'(0)\xi_n}{(1+\xi_n^2)^3}\nonumber\\
&-\frac{h'(0)}{4(1+\xi_n^2)^2}[\sum_{t<n}(\hat{c}(e_n)\hat{c}(e_t)-c(e_n)c(e_t))+2]\Big).
\end{align}
And the second item of \eqref{2a31},
\begin{align}
&\sigma_{1}(\widetilde{\nabla}_{U}\widetilde{\nabla}_{V})\sigma_{-2}(T^{-2})(x_0)|_{|\xi'|=1}\nonumber\\
=&\Big(\sqrt{-1}\sum_{j,l=1}^nU_j\frac{\partial_{U_l}}{\partial_{x_j}}\sqrt{-1}\xi_l
+\sqrt{-1}\sum_jA(V)U_j\xi_j+\sqrt{-1}\sum_lA(V)V_l\xi_l\nonumber\\
+&\sum_j \frac{1}{2}g(U,V')IdY_j\sqrt{-1}\xi_j+\sum_j \frac{1}{2}g(V,V')IdU_j\sqrt{-1} \xi_j\Big)\times|\xi|^{-2}.
\end{align}
Expand the last item  of \eqref{2a31} to have,
\begin{align}
&\sum_{j=1}^{n}\sum_{\alpha}\frac{1}{\alpha!}\partial^{\alpha}_{\xi}\big[\sigma_{2}(\widetilde{\nabla}_{U}\widetilde{\nabla}_{V})\big]
T_x^{\alpha}\big[\sigma_{-2}(T^{-2})\big](x_0)|_{|\xi'|=1}\nonumber\\
=&\sum_{j=1}^{n}\partial_{\xi_{j}}\big[-\sum_{j,l=1}^nU_jV_l\xi_j\xi_l\big]
(-\sqrt{-1})\partial_{x_{j}}\big[|\xi|^{-2}\big]\nonumber\\
=&\sum_{j=1}^{n}\sum_{l=1}^{n}\sqrt{-1}(x_{j}V_l+x_{l}V_j)\xi_{l}\partial_{x_{j}}(|\xi|^{-2}).
\end{align}
We note that $\int_{|\xi'|=1}\xi_{i_{1}}\xi_{i_{2}}\cdots\xi_{i_{2d+1}}\sigma(\xi')=0,i<n,$,
therefore, we omit some terms that do not contribute to the calculation of {\bf case c)}.
The following formula can be obtained by direct calculation
\begin{align}\label{71}
&\mathrm{Tr}[\pi^+_{\xi_n}\sigma_{-1}(\sigma_{2}(\widetilde{\nabla}_{U}\widetilde{\nabla}_{V})\sigma_{-3}(T^{-2}))\times
\partial_{\xi_n}\sigma_{-2}(T^{-2})](x_0)|_{|\xi'|=1}\nonumber\\
&=2^4\sum_{j,l=1}^{n-1}U_jV_l\xi_j\xi_l\times\Big(\frac{2i\xi_n-\xi_n^2}{2(\xi_n-i)^4(\xi_n+i)^2}\langle V', \xi'\rangle
+\frac{\xi_n}{2(\xi_n-i)^4(\xi_n+i)^2}\langle V', dx_n\rangle\nonumber\\
&+\left[\frac{3i\xi_n}{2(\xi_n-i)^3(\xi_n+i)^2}+\frac{3\xi_n}{2(\xi_n-i)^4(\xi_n+i)^2}-\frac{2i\xi_n^2+5\xi_n}{8(\xi_n-i)^5(\xi_n+i)^2}\right]h'(0)\Big)\nonumber\\
&+2^4U_nV_n\big( \frac{-\xi_n^2}{2(\xi_n-i)^4(\xi_n+i)^2}\langle V', \xi'\rangle
+\frac{3\xi_n^2-i\xi_n}{8(\xi_n-i)^4(\xi_n+i)^2}\langle V', dx_n\rangle
+\frac{i\xi_n^3+(3-6i)\xi_n^2+\xi_n}{16(\xi_n-i)^5(\xi_n+i)^2}h'(0) \big),
\end{align}
 
\begin{align}\label{72}
&\mathrm{Tr}[\pi^+_{\xi_n}\sigma_{-1}(\sigma_{1}(\widetilde{\nabla}_{U}\widetilde{\nabla}_{V})\sigma_{-2}(T^{-2}))\times
\partial_{\xi_n}\sigma_{-2}(T^{-2})](x_0)|_{|\xi'|=1}\nonumber\\
&=2^4U_n\frac{\partial V_n}{\partial x_n} \frac{i\xi_n-\xi_n^2}{(\xi_n-i)^4(\xi_n+i)^2}
+2^4\left[\langle U,V'\rangle V_n+\langle V,V'\rangle U_n\right]\frac{i\xi_n-\xi_n^2}{2(\xi_n-i)^4(\xi_n+i)^2},
\end{align}
and
\begin{align}\label{73}
\mathrm{Tr}[\pi^+_{\xi_n}\sigma_{-1}(\sum_{j=1}^{n}\sum_{\alpha}\frac{1}{\alpha!}\partial^{\alpha}_{\xi}
\big[\sigma_{2}(\widetilde{\nabla}_{U}\widetilde{\nabla}_{V})\big]
T_x^{\alpha}\big[\sigma_{-2}(T^{-2})\big])\times
\partial_{\xi_n}\sigma_{-2}(T^{-2})](x_0)|_{|\xi'|=1}
=U_nV_nh'(0)\frac{-4i\xi_n^2}{(\xi_n^2+1)^2}.
\end{align}
Substituting \eqref{71}, \eqref{72} and \eqref{73} into \eqref{161} yields
\begin{align}\label{74}
\Phi_5=&\Big([+\frac{4\pi\langle dx_n,V'\rangle}{3}+\frac{(96-17i)\pi h'(0)}{12}]\sum_{j=1}^{n-1}U_jV_j
-[\frac{ i\langle dx_n,V'\rangle}{4}+(\frac{3+6i}{32}-6i) h'(0)]U_nV_n \Big) \pi\Omega_3dx'\nonumber\\
&-(U_n\frac{\partial V_n}{\partial x_n}2\pi+\left[\langle U,V'\rangle V_n+\langle V,V'\rangle U_n\right])\pi \Omega_3dx'.
\end{align}
Let $U=U^T+U_n\partial_n,$ $V=V^T+V_n\partial_n$, then we have $\sum_{j=1}^{n-1}U_jV_j=g(U^T, V^T)$. Now we sum  the cases (a), (b) and (c) to get $\Phi$,
\begin{align}\label{795}
\Phi=&\sum_{i=1}^5\Phi_i=\Big([\frac{23-130i}{32}+\frac{5i}{4}\langle dx_n, V'\rangle]U_nV_n+\frac{(142+23i)\pi}{24}\langle U^T, V^T\rangle\Big)h'(0)\pi\Omega_3dx'\nonumber\\
&-(U_n\frac{\partial V_n}{\partial x_n}2\pi+\left[\langle U,V'\rangle V_n+\langle V,V'\rangle U_n\right])\pi \Omega_3dx'.
\end{align}
Then we obtain following theorem
\begin{thm}\label{thmb1}
 Let $M$ be a 4-dimensional compact manifold with boundary. Then we get the spectral Einstein functional  associated to $\widetilde{\nabla}_{U}\widetilde{\nabla}_{V}T^{-2}$
and $T^2$ on compact manifolds with boundary
\begin{align}
\label{b2773}
&\widetilde{{\rm Wres}}[\pi^+(\widetilde{\nabla}_{U}\widetilde{\nabla}_{V}T^{-2})\circ\pi^+T^{-2}]\nonumber\\
=&\frac{4\pi^{2}}{3}\int_{M}\big(Ric(U,V)-\frac{1}{2}sg(U,V)\big) vol_{g}
-\int_{M}\big(8\pi^2[g(V, \nabla_{U}V')+g(U,\nabla_{V}V')]-32s+4|V'|^2g(U,V)\big) vol_{g}\nonumber\\
+&\Big([\frac{23-130i}{32}+\frac{5i}{4}\langle dx_n, V'\rangle]U_nV_n+\frac{(142+23i)\pi}{24}\langle U^T, V^T\rangle\Big)h'(0)\pi\Omega_3dx'\nonumber\\
-&(U_n\frac{\partial V_n}{\partial x_n}2\pi+\left[\langle U,V'\rangle V_n+\langle V,V'\rangle U_n\right])\pi \Omega_3dx'.
\end{align}
\end{thm}

 \section{Residue for  statistical de Rham Hodge operators
 with torsion  $\widetilde{\nabla}_{U}\widetilde{\nabla}_{V}T^{-1}$ and $T^{-3}$ }

In this section, we compute the 4-dimension volume for   statistical de Rham Hodge operators
   $\widetilde{\nabla}_{U}\widetilde{\nabla}_{V}T^{-1}$ and $T^{-3}$ .
Since $[\sigma_{-4}(\widetilde{\nabla}_{U}\widetilde{\nabla}_{V}T^{-2}
\circ T^{-2})]|_{M}$ and $[\sigma_{-4}(\widetilde{\nabla}_{U}\widetilde{\nabla}_{V}T^{-1}
\circ T^{-3})]|_{M}$ have
the same expression in the case of manifolds without boundary, so we only need to calculate $\int_{\partial M}\tilde{\Phi}$.

Similar definition \ref{defn1}, we have
\begin{align}\label{a41}
	\tilde{\Phi}=&\int_{|\xi'|=1}\int_{-\infty}^{+\infty}\sum_{j,k=0}^{\infty}\sum \frac{(-i)^{|\alpha|+j+k+\ell}}{\alpha!(j+k+1)!}
	\mathrm{Tr}_{S(TM)}[\partial_{x_{n}}^{j}\partial_{\xi'}^{\alpha}\partial_{\xi_{n}}^{k}\sigma_{r}^{+}
	((\widetilde{\nabla}_{U}\widetilde{\nabla}_{V}T^{-q_{1}}))(x',0,\xi',\xi_{n})\nonumber\\
	&\times\partial_{x_{n}}^{\alpha}\partial_{\xi_{n}}^{j+1}\partial_{x_{n}}^{k}\sigma_{l}((T^{3})^{-q_{2})}(x',0,\xi',\xi_{n})]
	\texttt{d}\xi_{n}\sigma(\xi')\texttt{d}x' ,
\end{align}
and the sum is taken over $r-k+|\alpha|+\ell-j-1=-n,r\leq-q_{1},\ell\leq-q_{2}$.

From lemma \ref{lem3} and lemma \ref{lem4}, we have
\begin{lem} The following identities hold:
\begin{align}
\sigma_{1}(\widetilde{\nabla}_{U}\widetilde{\nabla}_{V} T ^{-1})=&
-\sqrt{-1}\sum_{j,l=1}^nU_jV_l\xi_j\xi_lc(\xi)|\xi|^{-2};\\
\sigma_{0}(\widetilde{\nabla}_{U}\widetilde{\nabla}_{V}T^{-1})=&
\sigma_{2}(\widetilde{\nabla}_{U}\widetilde{\nabla}_{V})\sigma_{-2}(T^{-1})
+\sigma_{1}(\widetilde{\nabla}_{U}\widetilde{\nabla}_{V})\sigma_{-1}(T^{-1})\nonumber\\
&+\sum_{j=1}^{n}\partial _{\xi_{j}}\big[\sigma_{2}(\widetilde{\nabla}_{U}\widetilde{\nabla}_{V})\big]
D_{x_{j}}\big[\sigma_{-1}(T^{-1})\big].
\end{align}
\end{lem}

Hence by Lemma 3.7 in \cite{wswy}, we have
\begin{lem} \label{lem43}
	The symbol of the  Statistical de Rham Hodge operator:
\begin{align}\label{2000}
\sigma_{-3}(T^{-3})&=\sqrt{-1}c(\xi)|\xi|^{-4};\\
\sigma_{-4}(T^{-3})&=\frac{c(\xi)\sigma_2(T^{3})c(\xi)}{|\xi|^8}+\frac{\sqrt{-1}c(\xi)}{|\xi|^8}\bigg(|\xi|^4c(dx_n)\partial_{x_n}c(\xi')\nonumber\\
&-2h'(0)c(\mathrm{d}x_n)c(\xi)+2\xi_nc(\xi)\partial{x_n}c(\xi')+4\xi_nh'(0)\bigg),
\end{align}
where,
 \begin{align}
\sigma_2(T^{-3})&=c(\mathrm{d}x_l)\partial_l(g^{i,j})\xi_i\xi_j+c(\xi)(4\sigma^k+4a^k-2\Gamma^k)\xi_{k}-2[c(\xi)\iota(V')-|\xi|^2\iota(V')]\nonumber\\
     &-\frac{1}{4}|\xi|^2\sum_{s,t}\omega_{s,t}({e_l})[c(e_{l})\overline c({e_s})\overline c({e_t})-c(e_l)c(e_s)c(e_t)]+|\xi|^2\iota(V').
\end{align}
\end{lem}
Next we  make a calculation for $\int_{\partial M} \widetilde{\Phi}$. When $n=4$, we have ${\rm tr}_{\wedge^*T^*M}[{\rm \texttt{id}}]=2^4$, the sum is taken over $
r+l-k-j-|\alpha|=-3,~~r\leq 0,~~l\leq-2,$ then we have the following five cases:

 {\bf case a)} According to \cite{WWw}, we have
\begin{align}\label{35}
&case~a-I):	~~\widetilde{\Phi}_1=0;\nonumber\\
&case~a-II): ~\widetilde{\Phi}_2
=\left[\frac{29}{48}\pi\sum_{j=1}^{n-1}U_jV_j
+\left(\frac{149+120i}{256}\right)U_nV_n\right]h'(0)\pi\Omega_3dx';\nonumber\\
&case~a-III): \widetilde{\Phi}_3
=\left(\frac{10\pi i}{3}\sum_{j=1}^{n-1}U_jV_j+\frac{5i}{4}U_nV_n\right)h'(0)\pi\Omega_3dx'.
\end{align}

 {\bf case b)}~$r=0,~l=-3,~k=j=|\alpha|=0$.

By  \eqref{a41}, we get
\begin{align}\label{42}
\widetilde{\Phi}_4&=-i\int_{|\xi'|=1}\int^{+\infty}_{-\infty}
\mathrm{Tr}[\pi^+_{\xi_n}\sigma_{0}(\widetilde{\nabla}_{U}\widetilde{\nabla}_{V}T^{-1})\times
\partial_{\xi_n}\sigma_{-3}(T^{-3})](x_0)d\xi_n\sigma(\xi')dx'.
\end{align}
By Lemma \ref{lem43}, we obtain
\begin{align}\label{43}
\partial_{\xi_n}\sigma_{-3}(T^{-3})(x_0)|_{|\xi'|=1}
=\frac{ic(\mathrm{d}x_n)}{(1+\xi_n^2)^2}-\frac{4\sqrt{-1}\xi_nc(\xi)}{(1+\xi_n^2)^3}.
\end{align}
 Expend the symbol of $\widetilde{\nabla}_{U}\widetilde{\nabla}_{V}T^{-1}$, we have
\begin{align}\label{b41}
\sigma_{0}(\widetilde{\nabla}_{U}\widetilde{\nabla}_{V}T^{-1})=&
\sigma_{2}(\widetilde{\nabla}_{U}\widetilde{\nabla}_{V})\sigma_{-2}(T^{-1})
+\sigma_{1}(\widetilde{\nabla}_{U}\widetilde{\nabla}_{V})\sigma_{-1}(T^{-1})\nonumber\\
&+\sum_{j=1}^{n}\partial _{\xi_{j}}\big[\sigma_{2}(\widetilde{\nabla}_{U}\widetilde{\nabla}_{V})\big]
D_{x_{j}}\big[\sigma_{-1}(T^{-1})\big]\nonumber\\
&:=A+B+C.
\end{align}

(A) Explicit representation the first item of \eqref{b41}
\begin{align}
&\sigma_{2}(\widetilde{\nabla}_{U}\widetilde{\nabla}_{V})\sigma_{-2}(T^{-1})(x_0)|_{|\xi'|=1}
=-\sum_{j,l=1}^{n}U_jV_l\xi_j\xi_l\sigma_{-2}(T^{-1})(x_0)|_{|\xi'|=1}\nonumber\\
&=-\sum_{j,l=1}^{n-1}U_jV_l\xi_j\xi_l\sigma_{-2}(T^{-1})(x_0)|_{|\xi'|=1}
  -U_nV_n\xi_n^2\sigma_{-2}(T^{-1})(x_0)|_{|\xi'|=1}\nonumber\\
&-\sum_{j=1}^{n-1}U_jV_n\xi_j\xi_n\sigma_{-2}(T^{-1})(x_0)|_{|\xi'|=1}
-\sum_{l=1}^{n-1}U_nV_l\xi_n\xi_l\sigma_{-2}(T^{-1})(x_0)|_{|\xi'|=1},
\end{align}
we let
\begin{align}
	Q_1=\sum_{j,l=1}^{n-1}U_jV_l\xi_j\xi_l\sigma_{-2}(T^{-1})(x_0)|_{|\xi'|=1},~~
	Q_2=U_nV_n\xi_n^2\sigma_{-2}(T^{-1})(x_0)|_{|\xi'|=1}.
\end{align}
By \cite{wswy}, we have
\begin{align}
	&\mathrm{Tr}[Q_1\times\partial_{\xi_n}\sigma_{-3}T^{-3}]|_{|\xi'|=1}\nonumber\\
	&=\sum_{j,l=1}^{n-1}U_jV_l\xi_j\xi_l\left[8h'(0)\frac{3+12i\xi_n+3\xi_n^2}{(\xi_n-i)^4(\xi_n+i)^3}
	+8h'(0)\frac{4i-11\xi_n-6i\xi_n^2+3\xi_n^3}{(\xi_n-i)^5(\xi_n+i)^3}\right.\nonumber\\
		&+\left.\frac{-2-8i\xi_n+6\xi_n^2}{4(\xi_n-i)^2(\xi_n+i)^3}\mathrm{Tr}[\iota(V')(x_0)c(\xi')]
		+\frac{-2i+8\xi_n+6i\xi_n^2}{4(\xi_n-i)^4(\xi_n+i)^3}\mathrm{Tr}[c(dx_n)\iota(V')(x_0)]\right].
\end{align}
Refer to section 4 in \cite{wswy} for the detailed calculation process, we have,
\begin{align}
	&-i\int_{|\xi'|=1}\int^{+\infty}_{-\infty} \mathrm{Tr}[Q_1\times \partial_{\xi_n}\sigma_{-3}(T^{-3})](x_0)d\xi_n\sigma(\xi')dx'\nonumber \\
	&=\left[\frac{110}{3}h'(0)+\frac{36i-16}{3}\langle dx_n, V'\rangle \right]\sum_{j}U_jV_j\pi^2\Omega_3dx'.
\end{align}
By integrating formula, we obtain
\begin{align}
&\pi_{\xi_n}^+[\xi_n^2\sigma_{-2}(T^{-1})]
=\pi_{\xi_n}^+\left[\frac{\xi_n^2c(\xi)A(x_0)c(\xi)}{(1+\xi_{n}^2)^2}\right]
+\pi_{\xi_n}^+\left[\frac{\xi_n^2c(\xi)\iota(V')(x_0)c(\xi)}{(1+\xi_{n}^2)^2}\right]\nonumber\\
&+\pi_{\xi_n}^+\left[\frac{\xi_n^2c(\xi)B(x_0)c(\xi)+\xi_n^2c(\xi)c(dx_n)\partial_{x_{n}}[c(\xi')](x_0)}{(1+\xi_{n}^2)^2}
-\frac{\xi_n^2h'(0)c(\xi)c(dx_n)c(\xi)}{(1+\xi_{n}^2)^3}\right],
\end{align}
where,
\begin{align}\label{}
A(x_0)=\frac{1}{4}\sum_{i,s,t}\omega_{s,t}(e_i)c(e_i)\hat{c}(e_s)\hat{c}(e_t),
~~B(x_0)=-\frac{1}{4}\sum_{i,s,t}\omega_{s,t}(e_i)c(e_i)c(e_s)c(e_t).\\
\end{align}
 Here we have
\begin{align}
	&\mathrm{Tr}[c(\xi)A(x_0)c(\xi)c(dx_n)]=\mathrm{Tr}[c(\xi)B(x_0)c(\xi)c(dx_n)]=0;\nonumber\\
	&\mathrm{Tr}[c(\xi)A(x_0)c(\xi)c(\xi')]=\mathrm{Tr}[c(\xi)B(x_0)c(\xi)c(\xi')]=0;\nonumber\\
	&\mathrm{Tr}[c(\xi')\iota(V')c(\xi')c(dx_n)]=-\mathrm{Tr}[c(dx_n)\iota(V')c(dx_n)c(dx_n)]=8\langle V', dx_n\rangle;\nonumber\\
	&\mathrm{Tr}[c(\xi')\iota(V')c(dx_n)c(dx_n)]=-\mathrm{Tr}[c(dx_n)\iota(V')c(\xi')c(dx_n)]=-8\langle V',\xi'\rangle;\nonumber\\
	&\mathrm{Tr}[c(\xi')c(dx_n)\partial_{x_{n}}[c(\xi')]c(dx_n)]=-\mathrm{Tr}[c(\xi')c(\xi')\partial_{x_{n}}[c(\xi')]c(\xi')]=-8h'(0).
\end{align}
We note that $\int_{|\xi'|=1}\xi_{i_{1}}\xi_{i_{2}}\cdots\xi_{i_{2d+1}}\sigma(\xi')=0,~i<n$,
so we omit some items that have no contribution for computing {\bf case b)}, therefore, we have
\begin{align}\label{b42}
	&\mathrm{Tr}[Q_2\times \partial_{\xi_{n}}(T^{-3})]\nonumber\\
	&=\left\{\pi_{\xi_n}^+\left[\frac{\xi_{n}^2c(\xi)\iota(V')(x_0)c(\xi)}{(1+\xi_n^2)^2}+\frac{\xi_{n}^2c(\xi)c(dx_n)\partial_{x_n}[c(\xi')]}{(1+\xi_n^2)^2}
	-h'(0)\frac{\xi_{n}^2c(\xi)c(dx_n)c(\xi)}{(1+\xi_n^2)^2}\right]\right\}\nonumber\\
	&\times\left[\frac{-4i\xi_nc(\xi')}{(1+\xi_n^2)^3}+\frac{i(1-3\xi_n^2)c(dx_n)}{(1+\xi_n^2)^3}\right]\nonumber\\
	&=\frac{-18\xi_n^3+(12i+8)\xi_n^2+6\xi_n-4i}{(\xi_n-i)^5(\xi_n+i)^3}\langle V', \xi'\rangle
	+\frac{-22\xi_n^2+(4i+2)\xi_n}{(\xi_n-i)^5(\xi_n+i)^3}\langle V', dx_n\rangle\nonumber\\
	&+\frac{-3\xi_n^4+24\xi_n^3-(28i-1)\xi_n^2-16\xi_n+4i}{(\xi_n-i)^5(\xi_n+i)^3}h'(0)
	+\frac{16i\xi_n(2\xi_n-i)}{(\xi_n-i)^5(\xi_n+i)^3}.
\end{align}

Substituting \eqref{b42} into \eqref{42} yields
\begin{align}\label{39}
&-i\int_{|\xi'|=1}\int^{+\infty}_{-\infty}
\mathrm{Tr} [\pi^+_{\xi_n}Q_2\times
\partial_{\xi_n}\sigma_{-3}(T^{-3})](x_0)d\xi_n\sigma(\xi')dx'\nonumber\\
=&U_nV_n[\frac{2-7i}{2}\langle dx_n, V'\rangle+\frac{2i-3}{2}h'(0)]\pi\Omega_3dx'.
\end{align}

(B) Clearly indicate the second item of \eqref{b41}
\begin{align}
&\sigma_{1}(\widetilde{\nabla}_{U}\widetilde{\nabla}_{V})\sigma_{-1}(T^{-1})(x_0)|_{|\xi'|=1}\nonumber\\
=&\Big(\sqrt{-1}\sum_{j,l=1}^nU_j\frac{\partial_{V_l}}{\partial_{x_j}}\partial_{x_l}
+\sqrt{-1}\sum_jA(U)U_j\xi_j+\sqrt{-1}\sum_lA(V)V_l\xi_l\nonumber\\
&+\sum_j\frac{1}{2}g(U,V')V_j\xi_j+\sum_j\frac{1}{2}g(V,V')U_j\xi_j \Big)\frac{\sqrt{-1}c(\xi)}{|\xi|^{2}};
\end{align}
By integrating formula we get
\begin{align}
&\pi^+_{\xi_n}\left(\Big(\sum_j\frac{1}{2}g(U,V')V_j\xi_j+\sum_j\frac{1}{2}g(V,V')U_j\xi_j\Big)
\frac{\sqrt{-1}c(\xi)}{|\xi|^{2}}\right)\nonumber\\
=&\pi^+_{\xi_n}\left(\Big(\sum_j^{n-1}\frac{1}{2}g(U,V')V_j\xi_j+\sum_j^{n-1}\frac{1}{2}g(V,V')U_j\xi_j \Big)
\frac{\sqrt{-1}c(\xi)}{|\xi|^{2}}\right)\nonumber\\
&+\pi^+_{\xi_n}\left(\Big( \frac{1}{2}g(U,V')V_n\xi_n+\frac{1}{2}g(V,V')U_n\xi_n \Big)
\frac{\sqrt{-1}c(\xi)}{|\xi|^{2}}\right)\nonumber\\
=&\Big(\sum_j^{n-1}\frac{1}{2}g(U,V')V_j\xi_j+\sum_j^{n-1}\frac{1}{2}g(V,V')U_j\xi_j \Big)
\frac{ic(\xi')-c(dx_n)}{2(\xi_{n}-i)}\nonumber\\
&+\Big( \frac{1}{2}g(U,V')V_n\xi_n+\frac{1}{2}g(V,V')U_n\xi_n \Big)
\frac{-c(\xi')-ic(dx_n)}{2(\xi_{n}-i)}.
\end{align}
We note that $i<n,~\int_{|\xi'|=1}\xi_{i_{1}}\xi_{i_{2}}\cdots\xi_{i_{2d+1}}\sigma(\xi')=0$,
then
\begin{align}\label{39}
&\mathrm{Tr} \left(\pi^+_{\xi_n}\Big(\sigma_{1}(\widetilde{\nabla}_{U}\widetilde{\nabla}_{V})\sigma_{-1}(T^{-1})\Big)\times
\partial_{\xi_n}\sigma_{-3}(T^{-3})\right)(x_0)\nonumber\\
&=-2^4\Big( \frac{1}{2}g(U,V')V_n+\frac{1}{2}g(V,V')U_n\Big)\frac{\xi_n+4i\xi_n^2-3\xi_n^3}{2(\xi_n-i)^4(\xi+i)^3}.
\end{align}
Substitute into \eqref{42}, so we get
\begin{align}
	&-i\int_{|\xi'|=1}\int^{+\infty}_{-\infty}
	\mathrm{Tr} [\pi^+_{\xi_n}\sigma_{1}(\widetilde{\nabla}_{U}\widetilde{\nabla}_{V})\sigma_{-1}(T^{-1})\times
	\partial_{\xi_n}\sigma_{-3}(T^{-3})](x_0)d\xi_n\sigma(\xi')dx'\nonumber\\
	&=3\Big( \frac{1}{2}g(U,V')V_n+\frac{1}{2}g(V,V')U_n\Big)\pi\Omega_3dx'.
\end{align}
(C) Finally, we show the third item of \eqref{b41}
\begin{align}
\sum_{j=1}^{n}\sum_{\alpha}\frac{1}{\alpha!}\partial^{\alpha}_{\xi}\big[\sigma_{2}(\widetilde{\nabla}_{U}\widetilde{\nabla}_{V})\big]
D_x^{\alpha}\big[\sigma_{-1}(T^{-1})\big](x_0)|_{|\xi'|=1}
=\sum_{j=1}^{n}\sum_{l=1}^{n}\sqrt{-1}(U_{j}V_l+U_{l}V_j)\xi_{l}\partial_{x_{j}}(\frac{\sqrt{-1}c(\xi)}{|\xi|^{2}}).
\end{align}
For further calculation, we have
\begin{align}
&\pi^+_{\xi_n}\left(\sum_{j=1}^{n}\sum_{\alpha}\frac{1}{\alpha!}\partial^{\alpha}_{\xi}\big[\sigma_{2}(\widetilde{\nabla}_{U}
\widetilde{\nabla}_{V})\big]
T_x^{\alpha}\big[\sigma_{-1}T^{-1})\big]\right)\nonumber\\
=&\sum_{l=1}^{n-1}(U_{n}V_l+U_{l}V_n)\xi_{l}\Big(\frac{i\partial_{x_{n}}(c(\xi'))}{2(\xi_n-i)}
+h'(0)\frac{(-2-i\xi_n )c(\xi')}{4(\xi_n-i)^2}
-h'(0)\frac{ic(dx_n)}{4(\xi_n-i)^2}\Big)\nonumber\\
&+U_{n}V_n\Big(\frac{-\partial_{x_{n}}(c(\xi'))}{(\xi_n-i)}
+h'(0)\frac{(-i )c(\xi')}{2(\xi_n-i)^2}
-h'(0)\frac{-i\xi_n c(dx_n)}{2(\xi_n-i)^2}\Big).
\end{align}
After calculation, we can get
\begin{align}\label{39}
&-i\int_{|\xi'|=1}\int^{+\infty}_{-\infty}
\mathrm{Tr}\Big[\pi^+_{\xi_n}\Big(\sum_{j=1}^{n}\sum_{\alpha}\frac{1}{\alpha!}\partial^{\alpha}_{\xi}
\big[\sigma_{2}(\widetilde{\nabla}_{U}\widetilde{\nabla}_{V})\big]
D_x^{\alpha}\big[\sigma_{-1}(T^{-1})\big]\Big)
\times
\partial_{\xi_n}\sigma_{-3}(T^{-3})\Big](x_0)d\xi_n\sigma(\xi')dx'\nonumber\\
&=\frac{7-15i}{2}X_{n}V_n\pi h'(0)\Omega_3dx'.
\end{align}
Summing up (A), (B) and (C) leads to the desired equality
\begin{align}\label{41}
\widetilde{\Phi}_4
&=\Big(\left[\frac{110}{3}h'(0)+\frac{36i-16}{3}\langle dx_n, V'\rangle \right]\pi\sum_{j}U_jY_j
+U_nV_n[\frac{2-7i}{2}\langle dx_n,V'\rangle+\frac{4-13i}{2}h'(0)]\nonumber\\
&+3\Big( \frac{1}{2}g(U,V')V_n+\frac{1}{2}g(V,V')U_n\Big)\Big)\pi\Omega_3dx'.
\end{align}

 {\bf  case c)}~$r=1,~\ell=-4,~k=j=|\alpha|=0$.

By  \eqref{a41}, we get
\begin{align}\label{61}
\widetilde{\Phi}_5&=-\int_{|\xi'|=1}\int^{+\infty}_{-\infty}\mathrm{Tr} [\pi^+_{\xi_n}
\sigma_{1}(\widetilde{\nabla}_{U}\widetilde{\nabla}_{V}T^{-1})\times
\partial_{\xi_n}\sigma_{-4}T^{-3})](x_0)d\xi_n\sigma(\xi')dx'\nonumber\\
&=\int_{|\xi'|=1}\int^{+\infty}_{-\infty}\mathrm{Tr}
[\partial_{\xi_n}\pi^+_{\xi_n}\sigma_{1}(\widetilde{\nabla}_{U}\widetilde{\nabla}_{V}T^{-1})\times
\sigma_{-4}T^{-3})](x_0)d\xi_n\sigma(\xi')dx'.
\end{align}
By \eqref{2000}, we have
\begin{align}\label{62}
\sigma_{-4}(T^{-3})(x_0)|_{|\xi'|=1}=&\frac{c(\xi)\sigma_2(T^3)c(\xi)}{(\xi_n^2+1)^4}+\frac{ic(\xi)}{(\xi_n^2+1)^4}
\left[|\xi|^4c(dx_n)\partial_{x_n}c(\xi')-2h'(0)c(dx_n)c(\xi)\right.\nonumber\\
&+\left.2\xi_nc(\xi)\partial_{x_n}c(\xi')+4\xi_nh'(0)\right] ,
\end{align}
where
\begin{align}
	\sigma_{2}(T^3)=-\frac{5}{2}h'(0)\xi_nc(\xi)-\frac{1}{4}h'(0)|\xi|^2c(dx_n)-2c(\xi)\iota(V')c(\xi)+3|\xi|^2\iota(V').
\end{align}
Using the lemma \ref{le:32}, we calculate
\begin{align}\label{261}
&\partial_{\xi_n}\pi^+_{\xi_n}\sigma_{1}(\widetilde{\nabla}_{U}\widetilde{\nabla}_{V}T^{-1})(x_0)|_{|\xi'|=1}
=\frac{c(\xi')+ic(\mathrm{d}x_n)}{2(\xi_n-i)^2}\Sigma_{j,l=1}^{n-1}U_jV_l\xi_j\xi_l\nonumber\\
&-\frac{(2i\xi_n-1)c(\xi')-(i+2\xi_n)c(\mathrm{d}x_n)}{2(\xi_n-i)^2}U_nV_n
+\frac{ic(\xi')-c(\mathrm{d}x_n)}{2(\xi_n-i)^2}\Sigma_{j=1}^{n}[U_jV_n\xi_j+U_nV_j\xi_j].
\end{align}
We note that $\int_{|\xi'|=1}\xi_{i_{1}}\xi_{i_{2}}\cdots\xi_{i_{2d+1}}\sigma(\xi')=0,~i<n$,
therefore, we omit some terms that do not contribute to the calculation. {\bf case c)}.
Also, straightforward computations yield
\begin{align}\label{271}
&{\rm tr}\bigg[\partial_{\xi_n}\pi^+_{\xi_n}\sigma_{-1}(\widetilde{\nabla}_{U}\widetilde{\nabla}_{V}T^{-1})\times
\frac{c(\xi)\sigma_2(T^3)c(\xi)}{(\xi_n^2+1)^4}\bigg]\nonumber\\
=&\sum_{j,l=1}^{n-1}U_jV_l\xi_j\xi_l\left\{\frac{-2h'(0)[12\xi_n+12i\xi_n^2-i]}{(\xi_n-i)^5(\xi_n+i)^3}
+\frac{4\langle V', \xi'\rangle+4i\langle V', dx_n\rangle}{(\xi_n-i)^4(\xi_n+i)^2}-\frac{24\langle V', \xi'\rangle}{(\xi_n-i)^5(\xi_n+i)^3}\right\}\nonumber\\
&-U_nV_n\left\{\frac{20h'(0)[-\xi_n+(3i-1)\xi_n^2+(2+3i)\xi_n^3+2\xi_n^4]}{(\xi_n-i)^6(\xi_n+i)^4}
 +\frac{2h'(0)[i+3i\xi_n^2-2\xi_n^3]}{(\xi_n-i)^5(\xi_n+i)^3}\right.\nonumber\\
&+(16\langle V', \xi'\rangle+16i\langle V', dx_n\rangle)\frac{2i\xi_n^5-\xi_n^4+3i\xi_n^3-2\xi_n^2+2i\xi_n-1}{(\xi_n-i)^6(\xi_n+i)^4}\nonumber\\
&\left.-\frac{32\langle V', \xi'\rangle (2i\xi_n^3-\xi_n^2-2i\xi_n+1)}{(\xi_n-i)^5(\xi_n+i)^3} +\frac{32\langle V', dx_n\rangle (2\xi_n+i)}{(\xi_n-i)^4(\xi_n+i)^2} \right\},
\end{align}
and
\begin{align}\label{272}
&{\rm tr}\Big[\partial_{\xi_n}\pi^+_{\xi_n}\sigma_{-1}(\widetilde{\nabla}_{U}\widetilde{\nabla}_{V}T^{-1})\times
\frac{ic(\xi)}{(\xi_n^2+1)^4}
\left[|\xi|^4c(dx_n)\partial_{x_n}c(\xi')-2h'(0)c(dx_n)c(\xi)\right.\nonumber\\
&+\left.2\xi_nc(\xi)\partial_{x_n}c(\xi')+4\xi_nh'(0)\right] \Big](x_0)|_{|\xi'|=1}\nonumber\\
&=4ih'(0)\sum_{j,l=1}^{n-1}U_jV_l\xi_j\xi_l\left[\frac{1}{(\xi_n+i)(\xi_n-i)^4}+\frac{2\xi_n-16}{(\xi_n+i)^3(\xi_n-i)^5}
  +\frac{8\xi_n+24}{(\xi_n+i)^4(\xi_n-i)^6}\right]\nonumber\\
&+U_nV_n\left[\frac{4h'(0)[1-(2+3i)\xi_n]}{(\xi_n-i)^4(\xi_n+i)^2}-\frac{-2+(4-i)\xi_n-(8+8i)\xi_n^2+5i\xi_n^3+2\xi_n^4}{(\xi_n-i)^6(\xi_n+i)^4}8h'(0)\right].
\end{align}
From \eqref{261},\eqref{271} and \eqref{272} we get
\begin{align}\label{74}
\widetilde{\Phi}_5
=\Big(\big(-\frac{48+37i}{24}\pi h'(0)-\frac{4\pi i}{3}\langle dx_n, V'\rangle\big)\sum_{j=1}^{n-1}U_jV_j
+\big(\frac{173-51i}{32}h'(0)-\frac{3i}{2}\langle dx_n, V'\rangle\big)U_nV_n\Big)\pi\Omega_3dx'.
\end{align}

 Let $U=U^T+U_n\partial_n,$ $V=V^T+V_n\partial_n$, then we have $\sum_{j=1}^{n-1}U_jV_j=g(U^T, V^T)$. Now we add cases (a), (b) and (c) to get $\widetilde\Phi$,
\begin{align}\label{795}
\widetilde{\Phi}=\sum_{i=1}^5\widetilde{\Phi}_i
&=\Big(\left[(\frac{1693}{48}+\frac{43i}{8})h'(0)+\frac{32i-16}{3}\langle dx_n, V'\rangle \right]\pi\langle U^T, V^T\rangle\nonumber\\
&+U_nV_n[(1-5i)\langle dx_n,V'\rangle+(\frac{2045}{256}-\frac{51i}{8})h'(0)]\nonumber\\
&+3\Big( \frac{1}{2}g(U,V')V_n+\frac{1}{2}g(V,V')U_n\Big)\Big)\pi\Omega_3dx'.
\end{align}
Finally, we prove the theorem:
\begin{thm}\label{thmb1}
When $M$ is a 4-dimensional compact manifold with boundary, we get the  spectral Einstein functional associated to $\widetilde{\nabla}_{U}\widetilde{\nabla}_{V}T^{-1}$
and $T^{-3}$ on compact manifolds with boundary
\begin{align}
\label{b2773}
&\widetilde{{\rm Wres}}[\pi^+(\widetilde{\nabla}_{U}\widetilde{\nabla}_{V}T^{-1})
\circ\pi^+(T^{-3})]\nonumber\\
=&\frac{4\pi^{2}}{3}\int_{M}\big(Ric(U,V)-\frac{1}{2}sg(U,V)\big) vol_{g}
-\int_{M}8\pi^2[g(V, \nabla_{U}V')+g(U,\nabla_{V}V')]-\big( 32s+4|V'|^2\big)g(U,V)vol_{g}\nonumber\\
+&\Big(\left[(\frac{1693}{48}+\frac{43i}{8})h'(0)+\frac{32i-16}{3}\langle dx_n, V'\rangle \right]\pi\langle U^T, V^T\rangle
+[(1-5i)\langle \frac{\partial}{\partial_{x_n}},V'\rangle+(\frac{2045}{256}-\frac{51i}{8})h'(0)]U_nV_n\nonumber\\
+& \frac{3}{2}g(X,V')V_n+\frac{3}{2}g(V,V')U_n\Big)\pi\Omega_3dx'.
\end{align}
\end{thm}

\noindent {\small\textbf{Acknowledgements}} This work was supported by NSFC. 11771070. The authors thank the referee for his (or her) careful reading and helpful comments.\\
\noindent {\small\textbf{Data availability statement}} The authors confirm that the data supporting the findings of this study
are available within the article.

\section*{Declarations}
\noindent {\small\textbf{Conflict of interest}} The authors state that there is no conflict of interest.

\end{document}